\documentclass[reqno]{amsart}
\usepackage[utf8x]{inputenc}  

\usepackage{setspace}
\usepackage[left=2.8cm, right=2.8cm, bottom=2.5cm]{geometry}                 
\DeclareUnicodeCharacter{2212}{\ensuremath{-}}
\usepackage{graphicx} 
\usepackage{pgf,tikz}
\usetikzlibrary{arrows}
\usepackage{amssymb}
\usepackage{amsmath}
\usepackage{pdfsync}
\usepackage{mathrsfs}
\usepackage{hyperref} 
\usepackage{verbatim} 
\usepackage{epstopdf}
\usepackage{bbm} 
\usepackage[colorinlistoftodos,prependcaption,textsize=tiny]{todonotes}
\usepackage{xargs}
\usepackage{lmodern}
\usepackage{mleftright}
\usepackage{listings}
\def\O{\operatorname{O}}

\def\meas{\operatorname{meas}}
\def\dist{\operatorname{dist}}
\def\erf{\operatorname{erf}}

\newcommandx{\emanuel}[2][1=]{\todo[linecolor=green,backgroundcolor=green!25,bordercolor=black,#1]{#2}}

\newcommandx{\diogo}[2][1=]{\todo[linecolor=orange,backgroundcolor=orange!25,bordercolor=orange,#1]{#2}}

\newcommandx{\mateus}[2][1=]{\todo[linecolor=blue,backgroundcolor=blue!25,bordercolor=blue,#1]{#2}}

\newcommandx{\danger}[2][1=]{\todo[linecolor=red,backgroundcolor=red!25,bordercolor=blue,#1]{#2}}

\renewcommand{\d}{\text{\rm d}}

\newtheorem{theorem}{Theorem}
\newtheorem{corollary}[theorem]{Corollary}

\newtheorem{remark}[theorem]{Remark}
\newtheorem{proposition}[theorem]{Proposition}

\newtheorem{example}[theorem]{Example}

\makeatletter
\DeclareFontFamily{U}{tipa}{}
\DeclareFontShape{U}{tipa}{m}{n}{<->tipa10}{}
\newcommand{\arc@char}{{\usefont{U}{tipa}{m}{n}\symbol{62}}}%

\newcommand{\arc}[1]{\mathpalette\arc@arc{#1}}

\newcommand{\arc@arc}[2]{%
  \sbox0{$\m@th#1#2$}%
  \vbox{
    \hbox{\resizebox{\wd0}{\height}{\arc@char}}
    \nointerlineskip
    \box0
  }%
}
\makeatother

\numberwithin{equation}{section}

\allowdisplaybreaks

\newcommand{\intav}[1]{\mathchoice {\mathop{\vrule width 6pt height 3 pt depth  -2.5pt
\kern -8pt \intop}\nolimits_{\kern -6pt#1}} {\mathop{\vrule width
5pt height 3  pt depth -2.6pt \kern -6pt \intop}\nolimits_{#1}}
{\mathop{\vrule width 5pt height 3 pt depth -2.6pt \kern -6pt
\intop}\nolimits_{#1}} {\mathop{\vrule width 5pt height 3 pt depth
-2.6pt \kern -6pt \intop}\nolimits_{#1}}}

\newcommand{\intavl}[1]{\mathchoice {\mathop{\vrule width 6pt height 3 pt depth  -2.5pt
\kern -8pt \intop}\limits_{\kern -6pt#1}} {\mathop{\vrule width 5pt
height 3  pt depth -2.6pt \kern -6pt \intop}\nolimits_{#1}}
{\mathop{\vrule width 5pt height 3 pt depth -2.6pt \kern -6pt
\intop}\nolimits_{#1}} {\mathop{\vrule width 5pt height 3 pt depth
-2.6pt \kern -6pt \intop}\nolimits_{#1}}}

\title[Distribution of the zeros of polynomials near the unit circle]{Distribution of the zeros of polynomials near the unit circle}

\author[Mithun Kumar Das]{Mithun Kumar Das}
\address{ Mathematics Section, The Abdus Salam International Centre for Theoretical Physics, Str. Costiera, 11, 34151 Trieste, Italy\\
and \\
School of Mathematical Sciences, National Institute of Science Education and Research, A CI of 
		Homi Bhabha National Institute, Jatni, Khurda, 
		752050, India.}
\email{mdas@ictp.it, das.mithun3@gmail.com}

\begin{document}

\begin{abstract}
We estimate the number of zeros of a polynomial in $\mathbb{C}[z]$ within any small circular disc centered on the unit circle, which improves and comprehensively extends a result established by Borwein, Erd{\'e}lyi, and Littmann~\cite{BE1} in 2008. Furthermore, by combining this result with Euclidean geometry, we derive an upper bound on the number of zeros of such a polynomial within a region resembling a gear wheel. Additionally, we obtain a sharp upper bound on the annular discrepancy of such zeros near the unit circle.
Our approach builds upon a modified version of the method described in \cite{BE1}, combined with the refined version of the best-known upper bound for angular discrepancy of zeros of polynomials.
\end{abstract}
\maketitle 
\section{Introduction}
Let $P_n$ be a polynomial in $\mathbb{C}[z]$ of degree $n$. Our primary focus in this study is to understand the distribution of the zeros of $P_n$ on the complex plane, emphasizing their distribution within and around the unit disk.
One of the most extensively studied classes of polynomials is the family of unimodular polynomials denoted as $\mathcal{K}_n$. These polynomials are defined as follows:
\begin{align}\label{21/6/00:48}
\mathcal{K}_n:= \Big\{ P_n : P_n(z)= \sum_{k=0}^na_kz^k, \: a_k\in \mathbb{C}, \, |a_k|=1 \Big\}.
\end{align}
If we further restrict the coefficients to only $-1$ and $1$, such polynomials are called Littlewood polynomials, named in honor of Littlewood, who extensively studied various analytical properties (the mean value, the number of zeros, etc.) related to polynomials and power series with constrained coefficients lying on the complex unit circle $|z|=1$ (see \cite{L, L1}).
 The study of complex zeros of Littlewood polynomials and polynomials with coefficients $\{-1,\, 0,\, 1\}$ is an old subject, dating back to the pioneering work of Bloch and P{\'o}lya \cite{BP}. Their research established that, on average, such polynomials of degree $n$ have at most $\O(\sqrt{n})$ real roots.
 Modern investigations into the distribution of zeros often originate from the seminal result of  Erd\H{o}s and Tur{\'a}n\cite{ET}, which addresses the angular equidistribution of zeros and is applicable to polynomials within the $\mathbb{C}[z]$. In a later section, we will delve into a more detailed discussion of this matter.
 In their work \cite{BE}, Borwein and  Erd{\'e}lyi explore a broader perspective by considering a generalization of the class $\mathcal{K}_n$ as follows:
\begin{align}\label{21/6/15:48}
\mathcal{G}_n:= \Big\{ P_n : P_n(z)= \sum_{k=0}^na_kz^k, \: a_k\in \mathbb{C}, \, |a_0|=|a_n|=1,\, |a_k|\leq 1 \Big\}.
\end{align}
They demonstrate that any polygon with vertices on the unit circle $|z|=1$ contains at most $\O(\sqrt{n})$ zeros of such polynomials. Such a result has been established in \cite{ PS, PY} for random polynomials.

 In 2008,  Borwein, Erd{\'e}lyi, Littmann~\cite{BE1} studied the distribution of zeros for any polynomial in $\mathcal{G}_n$ that lies within small circles centered on the unit circle. Specifically, they established the following result: 
\begin{theorem}[Theorem~1, \cite{BE1}]\label{BE1}
Let $P_n \in \mathcal{G}_n$ be a polynomial of degree $n$.
 Then any disk of radius  $33 \pi \log{n}/\sqrt{n}, (\leq 1)$ with the center on the unit circle $|z|=1$ contains at least $8\sqrt{n}\log{n}$ zeros of $P_n$.
\end{theorem}
For random polynomials, Pritsker and Sola~\cite{PS}, as well as Pritsker and Yeager~\cite{PY}, have derived the expectation of the above result asymptotically. In their work \cite{PS}, they considered the polynomial $P_n(z) = \sum_{k=0}^n C_kz^k$ with coefficients that are iid complex random variables with absolutely continuous distribution and $\mathbb{E}(|C_0|^t) < \infty$ for any $t>0$. Then, they showed that the expected number of zeros inside any disc of radius $r<2$ with center on the unit circle $|z|=1$ is given by
\begin{align*}
\frac{2}{\pi}\arcsin(r/2) n+\O(\sqrt{n\log {n}}) \quad \mbox{ as } n
\rightarrow \infty.
\end{align*}

\section{New results} \label{sec2}
After the pioneering work of Borwein, Erd{\'e}lyi and Littmann~\cite{BE1} on the polynomial class $\mathcal{G}_n$, no further extensions of general polynomials in $\mathbb{C}[z]$ have been explored, except for the results mentioned above about to random polynomials. In this study, we establish a lower bound for the number of zeros of polynomials in $\mathbb{C}[z]$ situated within a circular disk, with its center located on the unit circle $\mathbb{T}$. Furthermore, we introduce a parametrization to the radius of this circular disk. Defining some notations regarding the $p$-norm of polynomials on the unit circle $\mathbb{T}$ is necessary to present our results. Let
\begin{align}\label{P_n}
P_n(z) = c_n\prod_{j=1}^n \big( z-\alpha_j\big) = \sum_{j=0}^{n}c_jz^j
\end{align}
be a polynomial of degree $n$, with $c_0 \neq 0$ and roots $\alpha_j=\rho_j \, e(\theta_j)$, $j=1, 2, \ldots, n$,  where $e(\theta):=e^{2\pi i \theta }$ and throughout the paper we use the same notation. For $0< p < \infty$, a $p$-norm of $P_n$ is defined on unit circle by
\[
\|P_n\|_{p}= \left(\frac{1}{2\pi}\int_\mathbb{T} |P_n\big(e^{i \theta}\big)\big|^p\,\d\theta\right)^{\frac{1}{p}} = \left(\int_{\mathbb{R}/\mathbb{Z}} |P_n\big(e( \theta)\big)\big|^p\,\d\theta\right)^{\frac{1}{p}}.
\]
Thus, for $p=2$, we get $\|P_n\|_{2}^2 = |c_{n}|^2 + \ldots + |c_0|^2 $. The $\sup$-norm of $P_n$on $\mathbb{T}$ is defined by $\|P_n\|_{\infty}= \sup_{\mathbb{T}} |P_n\big(e^{i \theta}\big)\big|$.  Furthermore, let  $E$ be the subset of $[0, 1]$ which is given by $E=\{\theta \in [0,1] : |P_n(e(\theta))| < 1\}$ and for $0<p\leq \infty$, defined the logarithmic $p$-norm by
\begin{align}\label{norm B_p}
B_p(P_n)= \begin{cases}
\log{\left(\frac{\|P_n\|_{\infty}}{\sqrt{|c_0c_n|}}\right)} & \mbox{ if } p=\infty,\\
(1-|E|) \log{\left(\frac{\|P_n\|_{p}}{\sqrt{|c_0c_n|}}\right)} + \frac{1}{ep} & \mbox{ if } 0<p<\infty.
\end{cases}
\end{align}
 Now  we state our new results as follows:
\begin{theorem}\label{thm2} Let $P_n(z) = \sum_{k=0}^n c_kz^k $ be a polynomial of degree $n$ such that $P_n(0)\neq 0$ and  $\theta$ be a fixed real number in $[\frac{1}{2},\, 1 ]$. Then any disk of radius $\gamma_n = 7{\frac{(2B_\infty(P_n))^\theta}{\sqrt{n}}} \leq 1$ with center on the unit circle $|z|=1$, contains at least $\sqrt{n}(2B_\infty(P_n))^{\theta}$ zeros of $P_n$.

Furthermore, if $p\in (0,\,\infty)$ and $P_n$ satisfy the conditions $|c_nc_0|\geq 1$ and $\|P_n\|_p\geq 1$, then any disk of radius $\gamma_n = 9{\frac{(2B_p(P_n))^\theta}{\sqrt{n}}} \leq 1$ with center on the unit circle $|z|=1$, contains at least $\sqrt{n}(2B_p(P_n))^{\theta}$ zeros of $P_n$.
\end{theorem}
The subsequent result presents an analog of Theorem~\ref{thm2} for the class of polynomials $P_n \in \mathcal{G}_n$.
\begin{theorem}\label{thm3}
Let  $P_n \in \mathcal{G}_n$ a polynomial and $\theta \in [\frac{1}{2},\, 1 ]$ be any fixed real number. Then any open disk of radius $\gamma_n = 9{\frac{(\log{n})^\theta}{\sqrt{n}}} \leq 1$ with centre on $|z|=1$ contains at least $\sqrt{n}(\log{n})^{\theta}$ zeros of $P_n$.
\end{theorem}
\begin{remark}
In Theorem~\ref{thm3}, if we set $\theta = 1$, we can draw a meaningful comparison with Theorem~\ref{BE1}. Notably, when the condition $\gamma_n \leq 1$ holds, we are permitted to work with polynomials of degree as small as $6170$, whereas Theorem \ref{BE1} imposes a strict lower limit of $2307256$ for the polynomial degree.
Furthermore, the disk's radius described in Theorem~\ref{thm3} is approximately $11.5$ times smaller than the radius mentioned in Theorem~\ref{BE1}. However, in our case, the number of zeros within the disk is only $1/8$-th of the number of zeros within the disk discussed in \cite{BE1}.

Further, if we decide to maintain the same radius, denoted by  $\gamma_n = 33\pi{\frac{\log{n}}{\sqrt{n}}}$, then the disk of radius $\gamma_n$ contains at least $31\sqrt{n}\log{n}$ zeros of $P_n$.
\end{remark}
Now, we look at an application of Theorem \ref{thm2} in order to estimate the number of zeros inside a certain gear wheel type region. In mathematics, a gear wheel (also known as a gear or cogwheel) is typically defined in the context of geometry and trigonometry. Gears are used to transmit motion and force from one rotating object to another, and they have a wide range of applications in mechanical engineering, physics, and other fields. A geometric construction for 
the gear wheel region we consider here is the following: 
Create a sizeable circular disk, denoted as $C$, and symmetrically position $T$ smaller circles along the boundary of $C$, maintaining a uniform width gap $d$ between any two adjacent smaller circles. After resolving any intersections that may occur during this process, the result is a gear with $T$ teeth, each tooth possessing a width of $d$. Please refer to Figure \ref{Gear wheel}, generated using Mathematica software for visual reference.
We prove the following result on zero distribution result of $P_n$ on the gear wheel.
\begin{theorem}\label{thm gear}
Let $P_n(z) = \sum_{k=0}^n c_kz^k $ be a polynomial of degree $n$ such that $P_n(0)\neq 0$. Set $\gamma_n = 7{\frac{(2B_\infty(P_n))^\theta}{\sqrt{n}}} \leq \frac{1}{2}$ with $\theta \in [\frac{1}{2},\, 1 ]$ and 
 $0\leq \delta<1$. Then the gear wheel region on the unit disc with approximately $\frac{\pi}{1+\delta} (\arcsin(\frac{\gamma_n}{2}\sqrt{4-\gamma_n}))^{-1}$ teeth and each of width $2\delta\arcsin \left(\frac{\gamma_n}{2}\sqrt{4-\gamma_n^2}\right)$, contain at most $n\Big( 1- \frac{0.97\pi}{7(1+\delta)}\Big) $ zeros of $P_n$. 
 
Further, if $P_n$ satisfy the conditions $|c_nc_0|\geq 1$ and $\|P_n\|_p\geq 1$ and if we assume $\gamma_n = 9{\frac{(2B_p(P_n))^\theta}{\sqrt{n}}} \leq \frac{1}{2}$ then the same gear wheel contains at most $n\Big( 1- \frac{0.97\pi}{9(1+\delta)}\Big) $ zeros of  $P_n$.
\end{theorem} 
\begin{remark}
If $P_n \in \mathcal{G}_n$ we can choose $\gamma_n = 9{\frac{(\log{n})^\theta}{\sqrt{n}}} \leq \frac{1}{2}$. If $\theta=1$, the inequality $\gamma_n \leq \frac{1}{2}$ is valid for $n\geq 35582$. Thus, the associated gear wheel contains at most $66\%$ zeros of the polynomial $P_n.$
\end{remark}

\begin{figure}
	\includegraphics[scale=0.35]{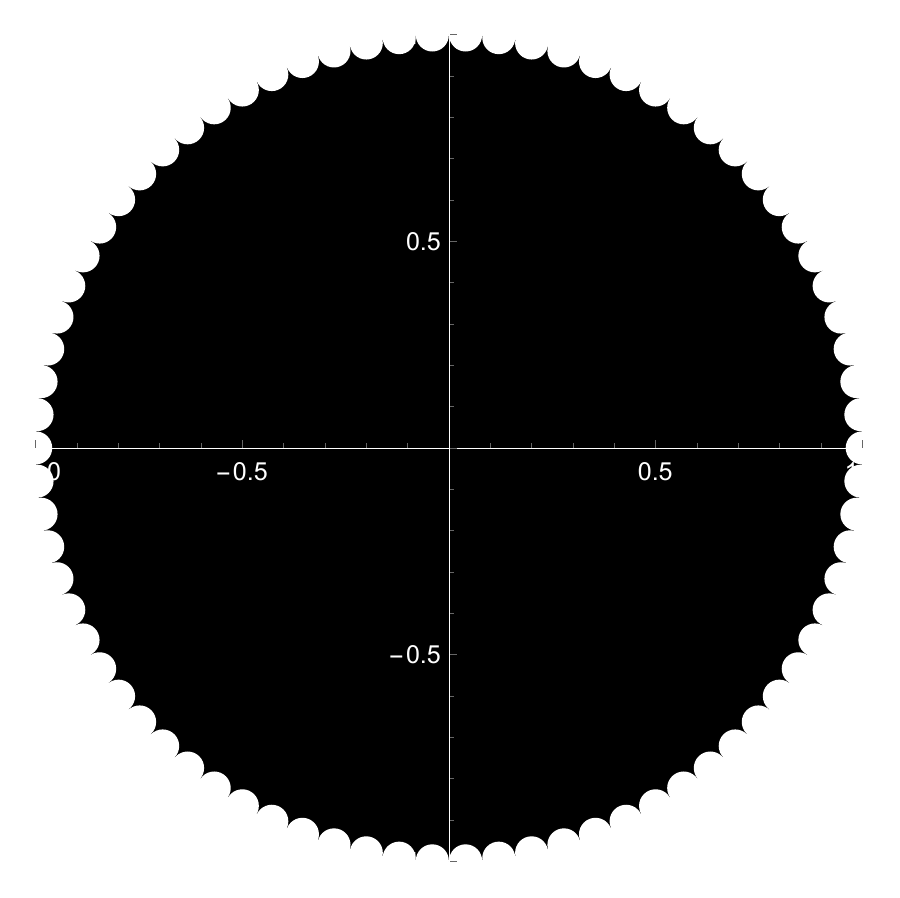}
	\includegraphics[scale=0.35]{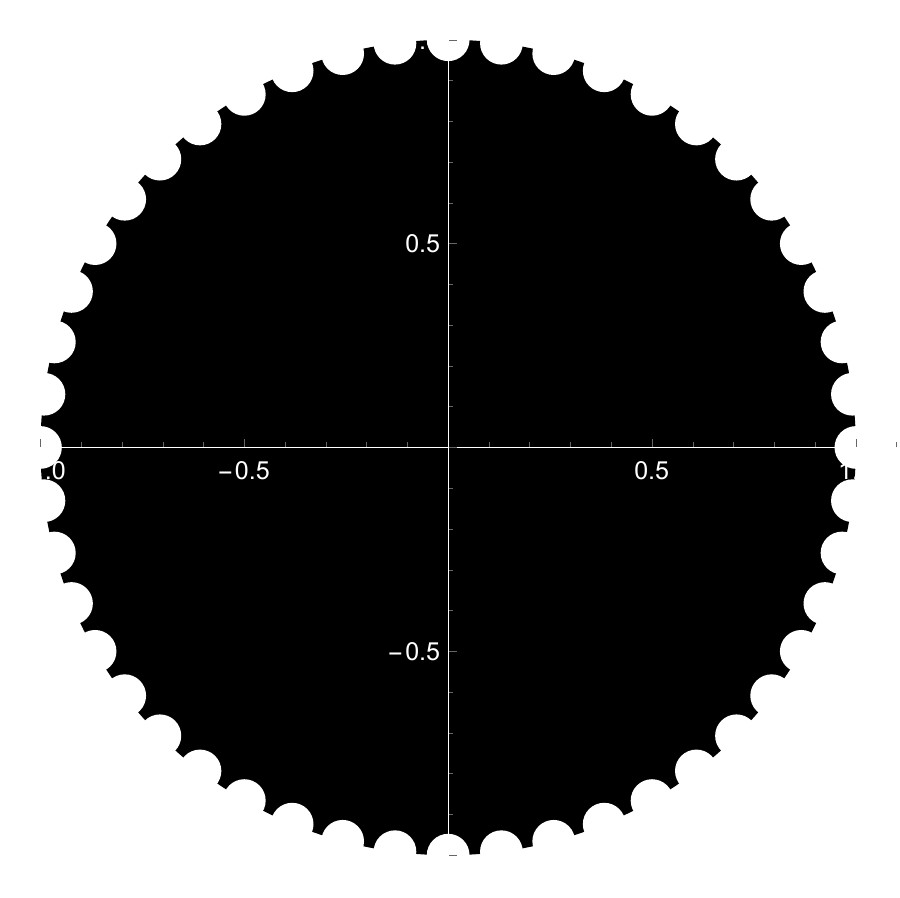}
	\caption{Left: Gear wheel of 78 teeth with width 0 ($\gamma_n\approx .04, \delta=0$); Right: Gear wheel of 48 teeth with width $\pi/120$ ($\gamma_n\approx \pi/60,  \delta \approx .2577$).}\label{Gear wheel}
\end{figure}

 \subsection{Height or mean of polynomials}\label{hight-norm section}\label{sub5}
 
 We consider here three notions of size, or {\it height}, of a polynomial $P_n(z)$ of degree $n$. The simple one is given by \[ \|P_n\|_{\infty} = \max_{|z| =1} \, |P_n(z)|. \]
  The Mahler measure (geometric mean) $M(P_n)$ of $P_n$ is defined as the following equivalent forms by using Jensen's formula,
\[
M(P_n)=\exp{\left(\frac{1}{2\pi}\int_{\mathbb{T}} \log\!\left(\big|P_n\big(e^{ i \theta}\big)\big| \right)\,\d\theta \right)}=\exp{\left(\int_{\mathbb{R}/\mathbb{Z}} \log\!\left(\big|P_n\big(e(\theta)\big)\big| \right)\,\d\theta \right)} = |c_n|\prod_{j=1}^n \max\{1, |\alpha_j|\}.
\]
The logarithmic Mahler measure, denoted by $m(P_n)$, is the natural logarithm of $M(P_n)$. Another one is defined as follows:
\[
M^+(P_n)=\exp{\left(\frac{1}{2\pi}\int_{\mathbb{T}} \log^+\!\left(\big|P_n\big(e^{ i \theta}\big)\big| \right)\,\d\theta \right)}=\exp{\left(\int_{\mathbb{R}/\mathbb{Z}} \log^+\!\left(\big|P_n\big(e(\theta)\big)\big| \right)\,\d\theta \right)},
\]
where $\log^+{x}= \max\{0,\, \log{x}\}$ and then set $m^+(P_n)= \log{M^+(P_n)}$. From the definition, one observed that 
\begin{equation}\label{1/9/15:00}
M(P_n)\leq M^{+}(P_n)\leq  \|P_n\|_{\infty}.
\end{equation}

For a monic polynomial $P_n$, the Mahler measure $M(P_n)\geq 1$ implies that the logarithmic Mahler measure $m(P_n) \geq 0$. A classical theorem of Kronecker implies that $M(P_n) = 1$ if and only if $P_n \in \mathbb{Z}[x]$  is a product of cyclotomic polynomials and the monomial $x$. An old question asks by Lehmer~\cite{Le} if there exist polynomials with Mahler measure between $1$ and $1 + \epsilon$ for arbitrary $\epsilon > 0$, and he notes that the polynomial
\[L(x) = x^{10} + x^9 − x^7 − x^6 − x^5 − x^4 − x^3 + x + 1 \]
has $M(L)=1.1762808 \ldots$. Several extensions have been done. But following \cite{BM, D, V} we have 

\begin{align}\label{lower bound of m(P)}
m(P_n)> \frac{1}{4}\Big( \frac{\log\, \log{n}}{\log{n}}\Big)^3
\end{align} 
for other general polynomials.

\subsection{ Discrepancy of zeros of polynomials}
This section will discuss a modified analysis of the angular distribution of the zeros of the polynomial $P_n(z)$ given in \eqref{P_n}.  This refinement serves as a fundamental tool for deriving the new results mentioned earlier. A standard way to define the normalized zero-counting measure for such a polynomial is
\[
\tau_n = \frac{1}{n}\sum_{k=1}^n\delta_{\alpha_k}. 
\]
A sector $S(\alpha,\, \beta)$ on $\mathbb{C}$ is defined as 
\begin{equation}\label{17/5/22:16}
S(\alpha,\, \beta) = \{z \in \mathbb{C} : 0<\alpha \leq  \arg{z} <\beta \leq 2\pi\}.
\end{equation}
Then, the discrepancy between the normalized zero counting measure in  $S(\alpha,\, \beta)$ and the Lebesgue measure of the arc on the unit circle in this sector is known as the angular discrepancy. It is denoted by $\mathcal{D}(P_n) $ and given by
\begin{align*}
\mathcal{D}(P_n) = \sup_{[\alpha,\beta]\subset \mathbb{T}}\left|\tau_n{(S(\alpha,\, \beta))} - \frac{\beta - \alpha}{2\pi}\right|.
\end{align*}
 In a pioneering work in 1950 concerning the angular discrepancy of zeros, Erd\H{o}s, and Tur\'an \cite{ET} proved that
\begin{equation}\label{0426_20:13}
\mathcal{D}(P_n) \le C \sqrt{\frac{B_\infty(P_n)}{n}}
\end{equation}
with $C = 16$. In 1954, Ganelius \cite{G} established \eqref{0426_20:13} with the constant $C = \sqrt{2\pi / k} =2.5619\ldots$, where $k=1/1^2 - 1/3^2 + 1/5^2 - \ldots =0.9159\ldots$ denotes the Catalan's constant. In 1992, Mignotte \cite{M} refined Ganelius's result by establishing the stronger inequality
\begin{equation}\label{0426_20:27}
\mathcal{D}(P_n) \le  C\sqrt{ \frac{1}{n} \,m^+\left(\frac{P_n}{\sqrt{|c_0c_n|}}\right)}.
\end{equation}
with the same constant $C = \sqrt{2\pi / k} =2.5619\ldots$. In 2019, Soundararajan \cite{S} improved this result by establishing \eqref{0426_20:27} with the constant $C = \frac{8}{\pi} = 2.5464\ldots$. Only recently, Carneiro et al.~\cite{CDFKMMTW}, where the author of this article is one of the co-authors, 
 improved the upper bound in \eqref{0426_20:27} by establishing the following result.

\begin{theorem}[Carneiro et al.~\cite{CDFKMMTW}, (2021)8]\label{C}
Let $P_n(z) = \sum_{k=0}^n c_kz^k $ be a polynomial of degree $n$ such that $P_n(0)\neq 0$. Then we have
\begin{equation*}
\mathcal{D}(P_n) \leq  \frac{4}{\sqrt{\pi}}\sqrt{ \frac{1}{n} \,m^+\left(\frac{P_n}{\sqrt{|c_0c_n|}}\right)}.
\end{equation*}
\end{theorem}
 However, more recently, in a preprint, Shu and Wang~\cite{SW} obtained a sharp discrepancy upper bound obtained in terms of the weaker norm. Precisely, they proved
\begin{theorem}[Shu-Wang(2021)~\cite{SW}]\label{SW1}
Let $P_n(z) = \sum_{k=0}^n c_kz^k $ be a polynomial of degree $n$ such that $P_n(0)\neq 0$. Then
 \begin{equation}\label{SW}
\mathcal{D}(P_n) \leq \sqrt{ \frac{2}{n} B_{\infty}(P_n)}.
\end{equation}
\end{theorem}  
   Note that in Theorem \ref{C} and Theorem \ref{SW1}, the computation of the asymptotic formulas for the associated norms is more complicated than the $p$-norms, $0<p<\infty$, and thus we need to use a trivial upper bound only. In the following proposition, we refine the Theorem \ref{C} in terms of $p$-norm.
\begin{proposition}\label{thm0}
 Let $p\in (0, \infty)$  and $P_n$ be a polynomial of degree $n$ such that $P_n(0)\neq 0$ and $\|P_n\|_p \geq 1$. Then
\begin{equation}\label{B_p}
\mathcal{D}(P_n) \leq \frac{4}{\sqrt{\pi}}\sqrt{\frac{B_p(P_n)}{n}}.
\end{equation}
\end{proposition}

Although, the leading constant in  the upper bound in Proposition~\ref{thm0} may be larger in compare to \eqref{SW}, but we have $\|P_n\|_p \leq \|P_n\|_{\infty}$.

Conditionally, we obtain the following explicit result with the distribution of $P_n$ on the unit circle.
\begin{corollary}\label{coro0}
Let $P_n \in \mathcal{K}_n$ such that  $\|P_n\|_p \geq 1$. Then 
\begin{equation*}
\left| \tau_n(S(\alpha,\, \beta)) - \frac{\beta - \alpha}{2\pi} \right| \leq \frac{2\sqrt{2}}{\sqrt{\pi}}\sqrt{\frac{(1-|E|)\log{(n+1)} + {e}^{-1}}{n}}.
\end{equation*}
Further, if $\theta$ is drawn randomly from $(0, 1)$, and $P_n(e(\theta))/\sqrt{n+1}$ converges in distribution (and in moments) to a standard complex Gaussian. Then, for sufficiently large $n$, we have
\begin{equation*}
\left| \tau_n(S(\alpha,\, \beta)) - \frac{\beta - \alpha}{2\pi} \right| < \frac{2\sqrt{2}}{\sqrt{\pi n}}\sqrt{\Big(1 -  \frac{1}{4}\erf^2{\Big(\sqrt{\frac{2}{n+1}}\Big)}\Big)\log{(n+1)} + e^{-1}},
\end{equation*}
where $\erf(x)= \frac{2}{\sqrt{\pi}}\int_{0}^xe^{-t^2}dt$.
\end{corollary}

Now, we look at a particular class of polynomials and find the effective upper bound of the discrepancy.
\begin{example}[Rudin–Shapiro polynomials]\label{exam} The Rudin–Shapiro polynomials are defined recursively, as follows: $P_0(z)=1=Q_0(z)$ and for $k=0,\, 1,\, 2,\, \ldots$, \[P_{k+1}(z)= P_k(z) + z^{2^k}Q_k(z)\quad \mbox{and}\quad  Q_{k+1}(z)= P_k(z) - z^{2^k}Q_k(z).\] We see that $P_k$ and $Q_k$ are two Littlewood polynomial of degree $n-1$, where $n:=2^k$. By the Parseval formula, $\|P_n\|_2=2^{k/2}$ .  By using this  we obtain a discrepancy upper bound in Proposition~\ref{thm0} \[\frac{2\sqrt{2}}{\sqrt{\pi}}\sqrt{\frac{(1-|E|)k\log{2}+ e^{-1}}{n-1}}.\] Recently, Rodgers~\cite{Ro} proved Saffari's conjecture and this gives that $|E|\sim (2)^{-(k+1)}$ as $k\rightarrow \infty$. Thus, for sufficiently large $k$, the discrepancy upper bound in Proposition~\ref{thm0} is 
\begin{equation}\label{26/9}
\left| \tau_n(S(\alpha,\, \beta)) - \frac{\beta - \alpha}{2\pi} \right| \leq \frac{2}{\sqrt{\pi}}\sqrt{\Big(2-\frac{1}{n-1}\Big)\frac{k\log{2}+ e^{-1}}{n-1}}.
\end{equation}
\end{example}
 
\subsection{Annular discrepancy of zeros of a polynomial} 
Our next result directly applies angular discrepancy, drawing inspiration from recent works~\cite{PS, Pr, PY}. We establish an upper bound for the discrepancy between the counting measure of zeros of the polynomial $P_n$ within an annular sector and the uniform measure on the unit circle $\mathbb{T}$. To provide context, let's define the annulus $\mathcal{A}_\rho$ and the annular sector $\mathcal{A}_\rho(\alpha, \beta)$ on the complex plane $\mathbb{C}$, where $0 < \rho < 1$:
\begin{equation}\label{17/5/22:18}
\mathcal{A}_\rho = \{z \in \mathbb{C} : \rho < |z| < 1/\rho \} \, \mbox{ and } \, \mathcal{A}_\rho(\alpha,\, \beta) = \{z \in \mathbb{C} : \rho < |z| < 1/\rho,\,  0<\alpha \leq  \arg{z} <\beta \leq 2\pi\}.
\end{equation}
Then, we obtain the following result.
\begin{theorem}\label{thm4}
Let $P_n(z) = \sum_{k=0}^n c_k z^k$ be a polynomial of degree $n$ such that $|c_n c_0| \geq 1$ and $\frac{\|P_n\|_\infty}{\sqrt{|c_0 c_n|}} \leq e^{\epsilon n}$ for sufficiently small $\epsilon > 0$. Then for sufficiently large $n$ (depending on $\epsilon$ and $\rho$), there exists a suitably small $\epsilon' > 0$ such that
\begin{equation*}
\left| \tau_n(\mathcal{A}_\rho(\alpha,\, \beta)) - \frac{\beta - \alpha}{2\pi} \right| \leq \left(\frac{4}{\sqrt{\pi}} +  \epsilon'\right)\sqrt{\frac{B_p(P_n)}{n}}.
\end{equation*}
Further, for $P_n \in \mathcal{G}_n$ and sufficiently large $n$ (depending on $\epsilon$ and $\rho$), there exists a suitably small $\epsilon''>0$ such that 
\begin{equation*}
\left| \tau_n(\mathcal{A}_\rho(\alpha,\, \beta)) - \frac{\beta - \alpha}{2\pi} \right| \leq \sqrt{\frac{\log{(n+1)}}{n}} (C+\epsilon''),
\end{equation*}
where $C$ is given by
\[C = \min\left\lbrace \sqrt{\frac{8}{\pi}( 1- |E|+ \frac{1}{e\log{(n+1)}}}, \, \sqrt{2}  \right\rbrace \leq 1.974\ldots . \]
\end{theorem}
Note that if $|E|<1-\pi/4$ then $C = \sqrt{2}$. The inequality $\frac{\|P_n\|_\infty}{\sqrt{|c_0c_n|}} \leq e^{\epsilon n}$ with suitably small $\epsilon$ is essentially the condition for angular equidistribution zeros of $P_n$ (see Section 1 in \cite{S}).
\section{Proof of results}
In order to prove the main results in Section \ref{sec2}, i.e., Theorem \ref{thm2}, \ref{thm3}, and \ref{thm gear}, we need to establish Proposition~\ref{thm0}, Theorem \ref{thm4} and Proposition \ref{lem2} at the beginning.

 \begin{proof}[Proof of Proposition~\ref{thm0}]
 Let us set $E^{c} = [0,1] \setminus E$. Then, by the hypothesis $\|P_n\|_p \geq 1$ implies that $|E^{c}| \neq \emptyset$.
 Thus, we write 
 \[ m^{+}(P_n) = \int_{E^{c}}\log{|P_n(e(\theta ))|} d\theta = \frac{|E^{c}|}{p}\left(\frac{1}{|E^{c}|}\int_{E^{c}}\log{|P_n(e(\theta ))|^p} d\theta \right). \]
 By using the concavity of the logarithmic function along with the Jensen's inequality on the integral above, we obtain 
 \[ m^{+}(P_n) \leq \frac{|E^{c}|}{p}\log\left(\frac{1}{|E^{c}|}\int_{E^{c}}{|P_n(e(\theta ))|^p} d\theta \right)= \frac{|E^{c}|}{p}\log{\frac{1}{|E^{c}|}} + (1-|E|) \log{\|P_n\|_p} . \]
 But $\max_{|E^{c}| \in (0, 1]}|E^{c}|\log{\frac{1}{|E^{c}|}}= \frac{1}{e}$, and hence, we have 
 \begin{equation}\label{21/9/22:00}
  m^{+}(P_n) \leq (1-|E|) \log{\|P_n\|_p} + \frac{1}{ep}.
 \end{equation}
The above argument is essentially a part of the proof \cite[Proposition~2.2]{PS}. We conclude the result by combining \eqref{21/9/22:00} and Theorem~\ref{C}.
 \end{proof}
  \begin{proof}[Proof of Corollary~\ref{coro0}]
  For $P_n \in \mathcal{K}_n$, we have $\log{|c_0c_n|}=0$ and $\|P_n\|_2 =\sqrt{n+1}$ and hence we obtain the first part of the corollary form Proposition~\ref{thm0}.
  
Since $P_n(e(\theta))/\sqrt{n+1}$ converges to a standard complex Gaussian, for sufficiently large $n$ we get
\begin{align*}
 \meas \left\{ \theta \in [0,\, 1] \,:\, \frac{1}{\sqrt{n+1}} P_n(e(\theta)) \in R \right\} = \frac{1}{\pi} \int_{R}e^{-x^2 - y^2}dx \, dy,
\end{align*}
where $R$ is any rectangle in $\mathbb{C}$. Let $R$ be a square of sides $\sqrt{\frac{2}{n+1}}$ and centered at the origin, then by simple Euclidean geometry, we have 
\[ |E| > \frac{1}{\pi}\int_{R}e^{-x^2 - y^2}dx \,dy = \frac{1}{4}\erf^2{\Big(\sqrt{\frac{2}{n+1}}\Big)}.\]
Hence, the result follows.
\end{proof}

\begin{proof}[Proof of Theorem~\ref{thm4}]

We recall the annulus $\mathcal{A}_\rho$ annular sector $\mathcal{A}_\rho(\alpha,\, \beta)$ from \eqref{17/5/22:18}. In the following proposition, we obtain an annular discrepancy upper bound of zeros of a polynomial, which is a better upper bound in comparison to \cite[Proposition 2.1]{PS}.
\begin{proposition}\label{lem2}
Let $P_n(z)= \sum_{k=0}^nc_kz^k$ be a polynomial of degree $n$ such that $|c_nc_0|\geq 1$ and for $p \in (0,\, \infty)$, recall $B_p(P_n)$ from Proposition~\ref{thm0}. For any $\rho \in (0,\, 1)$ and $0\leq \alpha < \beta < 2\pi $, we have 
\begin{equation}\label{20/5/19:09}
 \tau_n(\mathbb{C}\setminus \mathcal{A}_\rho) \leq \frac{2}{n(1 - \rho)} m\left(\frac{P_n}{\sqrt{|c_0c_n|}}\right) \leq \frac{2}{n(1 - \rho)}B_p(P_n)
\end{equation}
and 
\begin{equation}\label{20/5/19:10}
\left| \tau_n(\mathcal{A}_\rho(\alpha,\, \beta)) - \frac{\beta - \alpha}{2\pi} \right| \leq  \sqrt{\frac{2}{n}\min\left\lbrace \frac{8}{\pi} B_p(P_n), \, B_\infty(P_n) \right\rbrace} + \frac{2}{n(1 - \rho)}B_p(P_n) 
\end{equation}
\end{proposition} 

 Assume the mild condition $\frac{\|P_n\|_\infty}{\sqrt{|c_0c_n|}} \leq e^{\epsilon n}$ for suitably small $\epsilon>0$. Then we must find a suitably small $\epsilon'>0$ such that $\sqrt{B_p(P_n)/n} \leq \frac{1-\rho}{2} \epsilon'$. Specifically, we can choose $\epsilon'=\frac{2}{1-\rho}\sqrt{(1-|E|)\epsilon +\frac{1}{epn}}$, and this would be sufficiently small whenever $n\geq n_{\epsilon, \rho}$. Thus, for $p\in (0,\, \infty)$, \eqref{20/5/19:10} of Proposition~\ref{lem2} gives
\begin{align*}
\left| \tau_n(\mathcal{A}_\rho(\alpha,\, \beta)) - \frac{\beta - \alpha}{2\pi} \right| & \leq \sqrt{\frac{2}{n}\min\left\lbrace \frac{8}{\pi} B_p(P_n), \, B_\infty(P_n) \right\rbrace} + \epsilon' \sqrt{\frac{B_p(P_n)}{n}} \leq \left(\frac{4}{\sqrt{\pi}}+\epsilon'\right)\sqrt{\frac{B_p(P_n)}{n}}.
\end{align*}
For $p=2$ we have $\|P_n\|_2 \leq \sqrt{n+1}$ for $P_n \in \mathcal{G}_n$. This gives $B_2(P_n)= \frac{1}{2}((1-|E|)\log{(n+1)} + e^{-1})$. Also, we get the upper bound $B_{\infty}(P_n)\leq \log{(n+1)}$. Thus, \eqref{20/5/19:10} of Proposition~\ref{lem2} gives
 \begin{align*}
\left| \tau_n(\mathcal{A}_\rho(\alpha,\, \beta)) - \frac{\beta - \alpha}{2\pi} \right|& \leq \sqrt{\frac{\log{(n+1)}}{n}}\\
& \times \left( \min\left\lbrace \sqrt{\frac{8}{\pi}( 1- |E|+ (e\log{(n+1)})^{-1})}, \, \sqrt{2}    \right\rbrace  + \bigg(1- |E|+ \frac{1}{e\log{(n+1)}}\bigg)\epsilon'\right).
\end{align*}
Now, we choose $\epsilon''= \bigg(1- |E|+ \frac{1}{e\log{(n+1)}}\bigg)\epsilon'$ in above to complete the proof.
\end{proof}

\begin{proof}[Proof of Proposition \ref{lem2}]
Let $D_\rho$ be a disk which is defined by $D_\rho = \{z \in \mathbb{C}: |z|< \rho \}$ where $0< \rho < 1$. Then we have 
\begin{align*}
\tau_n(\mathbb{C}\setminus  \mathcal{A}_\rho) = \tau_n(\mathbb{C}\setminus D_{\frac{1}{\rho}}) + \tau_n( D_\rho).
\end{align*}
By appealing \cite[Lemma 5.1]{PS} in above equality we obtain the first inequality of \eqref{20/5/19:09}, i.e.,
\begin{equation}\label{19/5/15:48}
\tau_n(\mathbb{C}\setminus  \mathcal{A}_\rho) \leq \frac{2}{n(1-\rho)} m\left(\frac{P_n}{\sqrt{|c_0c_n|}}\right).
\end{equation}

To prove the second inequality of \eqref{20/5/19:09}, at first, we recall $E$ from Proposition~\ref{thm0} and we have \[m(P_n/\sqrt{|c_0c_n|})\geq 0.\] Then we obtain 
\begin{align*}
m\bigg(P_n/\sqrt{|c_0c_n|}\bigg) = \int_{E^c}\log{|P_n(e(\theta))|}d\theta +\int_{E}\log{|P_n(e(\theta))|}d\theta - \frac{1}{2}\log|c_0c_n|.
\end{align*}
But from the definition of $E$ we have $\int_{E}\log{|P_n(e(\theta))|}d\theta \leq 0$. Combining this with \eqref{21/9/22:00} and  using the condition $|c_nc_0|\geq 1$, we get
\begin{equation}\label{15/10}
m(P_n/\sqrt{|c_0c_n|}) \leq \int_{E^c}\log{|P_n(e(\theta))|}d\theta  - \frac{1 -|E|}{2}\log|c_0c_n| \leq B_p(P_n),
\end{equation}
for any $p\in (0,\, \infty)$. This completes the proof of \eqref{20/5/19:09}.

To prove the inequality \eqref{20/5/19:10}, we are following the same strategy as in the proof of Proposition 2.1 in \cite{PS}. That is, we write
\[\tau_n( \mathcal{A}_\rho(\alpha , \beta)) = \tau_n(S(\alpha , \beta)) - \tau_n(S(\alpha,\, \beta)\setminus  \mathcal{A}_\rho(\alpha , \beta)).\]
By using triangle inequality and the inequality $\tau_n(S(\alpha,\, \beta)\setminus  \mathcal{A}_\rho(\alpha , \beta)) \leq \tau_n(\mathbb{C}\setminus  \mathcal{A}_\rho)$, and then combining \eqref{19/5/15:48} with Proposition~\ref{thm0} and \eqref{SW}, we obtain the result.
\end{proof}

\begin{proof}[Proof of Theorem~\ref{thm2}]
 For the polynomial $P_n$, following \eqref{SW} we get
 \begin{align}\label{24/6/13:10}
 \Big| n\tau_n(S(\alpha, \, \beta)) - \frac{\beta - \alpha}{2\pi}n \Big| \leq \sqrt{2nB_{\infty}(P_n)},
 \end{align}
 where $B_{\infty}(P_n)= \log{\left(\frac{\|P_n\|_{\infty}}{\sqrt{|c_0c_n|}}\right)}$.
For $0 <\alpha < 1$, Proposition~\ref{lem2} gives 
\begin{align}\label{31/7/12:54}
n\tau_n(\mathbb{C}\setminus \mathcal{A}_{1-\alpha}) \leq \frac{2}{\alpha} m\left(\frac{P_n}{\sqrt{|c_0c_n|}}\right)\leq \frac{2}{\alpha} B_\infty(P_n).
\end{align}
Now, set $\delta_n= \frac{\pi a}{\sqrt{n}}\left(2B_\infty(P_n)\right)^\theta$, where $\theta \in \big[\frac{1}{2},\, 1 \big]$, $a>1$ is a positive real number and  $z_0 = e^{i\varphi}$ for some $\varphi \in [0, \, 2\pi )$. From \eqref{24/6/13:10} we have 
\begin{align*}
n\tau_n(\mathbb{C} \setminus S(\varphi-\delta_n,\, \varphi + \delta_n)) \leq n - \frac{\delta_n n}{\pi} + \sqrt{2nB_\infty(P_n)} \leq  n - (a -1 )\sqrt{n}\left(2B_\infty(P_n)\right)^{\theta}.
\end{align*}  
The sector $S(\varphi-\delta_n,\, \varphi + \delta_n)$ contains atleast $(a-1)\sqrt{n}(2B_\infty(P_n))^{\theta}$ zeros of $P_n$. Thus, the number of zeros in $\mathcal{A}_{1-\alpha} \cap S(\varphi-\delta_n,\, \varphi + \delta_n) $ is atleast
\begin{align*}
(a-1)\sqrt{n}(2B_\infty(P_n))^{\theta} - \frac{2}{\alpha}B_\infty(P_n).
\end{align*} 
Taking $ \alpha =\frac{b}{\sqrt{n}}(2B_\infty(P_n))^{1-\theta} $ on the above we get 
\[n\tau_n(\mathcal{A}_{1-\alpha} \cap S(\varphi-\delta_n,\, \varphi + \delta_n)) \geq  (a-b-1)\sqrt{n}(2B_\infty(P_n))^{\theta}. \]
Without lose of generality, if we take  $z_0 = 1$, we find a circle of radius \[r = \max\{\dist\big(1, (1-\alpha)e^{i\delta_n}\big),\, \dist\big(1, (1-\alpha)^{-1}e^{i\delta_n}\big) \},\] which contains the intersection $\mathcal{A}_{1-\alpha} \cap S(-\delta_n,\, \delta_n)$ and simple calculation gives 
\begin{align*}
r &= \max\Big\{ \Big(\frac{4\sin^2(\delta_n/2)}{1-\alpha} + \frac{\alpha^2}{(1-\alpha)^2}\Big)^{\frac{1}{2}}, \Big(\alpha^2 + 4(1-\alpha)\sin^2(\delta_n/2) \Big)^{\frac{1}{2}}\Big\}\\
 &=\Big(\frac{4\sin^2(\delta_n/2)}{1-\alpha} + \frac{\alpha^2}{(1-\alpha)^2}\Big)^{\frac{1}{2}}\\
 & \leq \Big(\frac{\delta_n^2}{1-\alpha} + \frac{\alpha^2}{(1-\alpha)^2}\Big)^{\frac{1}{2}}.
\end{align*}
The last inequality uses the inequality $\sin{x}\leq x$ for any non-negative $x$  with $0\leq x \leq 1/2$. 
Since $\delta_n \leq 1$ and $1-\theta \leq \theta$, the definition of $\delta_n$ gives \[\alpha \leq \frac{b}{\sqrt{n}}(2B_\infty(P_n))^{\theta} \leq \frac{b}{\pi a}\] and hence we get 
\begin{align}\label{31/8/00:39}
r \leq \frac{1}{\sqrt{n}}(2B_\infty(P_n))^\theta \pi a\sqrt{\frac{\pi a}{\pi a-b} + \frac{b^2}{(\pi a-b)^2}}. 
\end{align}
Note that if we choose $a=2.195$ and $b=0.195$ then we have $a-b-1=1$ and the inequality \eqref{31/8/00:39} reduce to
\[ r < \frac{7}{\sqrt{n}}(2B_{\infty}(P_n))^\theta= \gamma_n. \]
\newblock{Second Part.}
For $0 <\alpha < 1$ and $p\in (0, \infty)$, Proposition~\ref{lem2} gives
\begin{equation}\label{10/12}
n\tau_n(\mathbb{C}\setminus \mathcal{A}_{1-\alpha}) \leq \frac{2}{\alpha}B_p(P_n).
\end{equation}
 Along with \eqref{10/12}, the remaining proof follows from the first part with a certain change in the choice of parameters. In the first part, if we work with Proposition~\ref{thm0} in the place of \eqref{24/6/13:10} and \eqref{10/12} in the place of \eqref{31/7/12:54} and then proceeding further with the choice $\delta_n= \frac{\pi a}{\sqrt{n}}\left(2B_p(P_n)\right)^\theta$ and $ \alpha =\frac{b}{\sqrt{n}}(2B_p(P_n))^{1-\theta} $ we get a circular disk with radius $r$ contains atleast $(a-b-2\sqrt{2/\pi})\sqrt{n}(2B_p(P_n))^{\theta}$. Here $r$ is bounded by
 \begin{equation*}\label{31/8/00:40}
r \leq \frac{1}{\sqrt{n}}(2B_p(P_n))^\theta \pi a\sqrt{\frac{\pi a}{\pi a-b} + \frac{b^2}{(\pi a-b)^2}}. 
\end{equation*}
 Now, we choose $a=2.8$, $ b=1.8-2\sqrt{2/\pi}=0.20423\ldots$ such that $a-b-2\sqrt{2/\pi}=1$ and we have 
 \begin{equation*}
r < \frac{9}{\sqrt{n}}(2B_p(P_n))^\theta =\gamma_n . 
\end{equation*}
This completes the proof. 
\end{proof}

\begin{proof}[Proof of Theorem~\ref{thm3}]
We follow the proof of Theorem~\ref{thm2} with some adjustments
 to derive this result.
 Since $P_n \in \mathcal{G}_n$ we have $P_n(z) = \sum_{k=0}^na_kz^k$ with $|a_0| = |a_n| =1$ and $|a_k| \leq 1$. Thus, $\log{|a_0a_n|} =0$ and hence
\begin{equation}\label{30/8/23:50}
\log{ \| P_n \|_2 } = \frac{1}{2}\log{\left( \int_{0}^1|P_n(e(k \theta ))|^2d\theta \right)} = \frac{1}{2}\log{\left( \sum_{k=0}^n|a_k|^2  \right)} \leq   \frac{1}{2}\log{( n+1)}.
\end{equation}
For any polynomial $P_n \in \mathcal{G}_n$, by \eqref{30/8/23:50} we obtain a trivial bound \[B_2(P_n) \leq \frac{1}{2}\log{(n+1)} + \frac{1}{2e}.\] The proof follows from the proof of the second part in Theorem~\ref{thm2} with a few modifications. We replace the upper bounds of $B_2(P_n)$ in \eqref{10/12} and Proposition~\ref{thm0} by $\frac{1}{2}\log{(n+1)} + \frac{1}{2e}$ and will use for our purpose. Furthermore, we set the parameters $\delta_n$ and $\alpha$ as follows:
\[\delta_n= \frac{2.8 \pi}{\sqrt{n}}(\log(n+1) + 1/e)^\theta \mbox{  and }  \alpha =\frac{0.2}{\sqrt{n}}(\log{(n+1)} + 1/e)^{1-\theta} .\]
 With such changes, following the proof of Theorem~\ref{thm2} and using the inequality  $\log{(n+1)} > \log{n}$, we obtain the required result.
\end{proof}
\begin{proof}[Proof of Theorem \ref{thm gear}]
At first, we want to compute the length of an arc of the unit circle that lies inside the circle of radius $\gamma_n$, whose center is at some point of the unit circle. Let $O_1$ be such point and $O$ be the center of the unit circle. Also, denote the intersecting points of two circles by $A$ and $B$. Let $\theta$ denote the angle $\angle {AOB}$ and $\Gamma$ be the arc length of $AB$ of the unit circle. Since $O_1$ is the middle point of the arc $AB$ on the unit circle, we have $\angle {O_1OA}= {\theta}/{2}$. Thus, the arc length of $O_1A=\theta/2$. Now, the area of the triangle, $\triangle O_1OA = \frac{1}{2}\sin{(\theta/2)}$, since the side $O_1A=\gamma_n$. On the other hand, by the Heron's formula, we get
\begin{align*}
\triangle O_1OA = [s(s-1)^2(s-\gamma_n)]^{\frac{1}{2}}
\end{align*}
where $s$ is the semi-perimeter of the triangle, i.e., $s= 1+\gamma_n/2$. As a result, we get
\begin{align*}
\triangle O_1OA = \frac{\gamma_n}{2}\left(1-\frac{\gamma_n^2}{4}\right)^{\frac{1}{2}} =\frac{1}{2} \sin{(\theta/2)}.
\end{align*}
Since the $\Gamma = 1\cdot\theta = \theta$, we get 
\begin{align*}
\Gamma = 2\arcsin \left(\frac{\gamma_n}{2}\sqrt{4-\gamma_n^2}\right).
\end{align*}
Now, in order to make the unit disk a gear wheel, we have to draw circles consecutively of radius $\gamma_n$ with the center at the boundary of the unit disk such that the gap between any two consecutive small circles is $\delta \Gamma$, for some $\delta\in [0,1]$. Then, we have to remove the region of the disk that is inside of each of the small circles.
Thus, the total number of the teeth of the gear wheel equals the total number of gaps between consecutive circles drawn with the center at the boundary of the unit disk. Let $G$ be the such number.
Then 
\begin{align*}
G= \frac{2\pi}{\Gamma(1+\delta)}=\frac{\pi}{1+\delta} \left(\arcsin \left(\frac{\gamma_n}{2}\sqrt{4-\gamma_n^2}\right)\right)^{-1}.
\end{align*}
By Theorem \ref{thm2}, the number of zeros of $P_n(z)$ inside the gear wheel is at most 
\begin{align}\label{atmost}
n-G \sqrt{n}(2B(P_n))^{\theta} =n-\frac{\pi}{1+\delta}\sqrt{n}(2B(P_n))^{\theta}\left(\arcsin \left(\frac{\gamma_n}{2}\sqrt{4-\gamma_n^2}\right)\right)^{-1},
\end{align}
where $\gamma_n= \frac{9}{\sqrt{n}}(2B(P_n))^{\theta} \leq \frac{1}{2} $ and width of each tooth is
\begin{align*}
2\delta\arcsin \left(\frac{\gamma_n}{2}\sqrt{4-\gamma_n^2}\right).
\end{align*}
A sharp version of Shafer-Fink's inequality to bounding $\arcsin$-function by Male\v{s}evi\'c in \cite[Theorem 1.3]{BJM} gives
\begin{align*}
\arcsin{x}\leq \frac{\pi x}{2+(\pi-2)\sqrt{1-x^2}}, \quad 0\leq x\leq 1.
\end{align*}
Using this inequality in \eqref{atmost}, we get the number of zeros of $P_n(z)$ inside the gear wheel at most 
\begin{align*}
& n-\frac{\pi}{1+\delta}\frac{\sqrt{n}(2B(P_n))^{\theta}}{\gamma_n} \left(1-\frac{\pi -2}{2\pi}\gamma_n^2\right)\left(\sqrt{1-\frac{\gamma_n^2}{4}}\right)^{-\frac{1}{2}}\\
& = n - \frac{\pi n}{7(1+\delta)} \left(1-\frac{\pi -2}{2\pi}\gamma_n^2\right)\left(\sqrt{1-\frac{\gamma_n^2}{4}}\right)^{-\frac{1}{2}},
\end{align*}
where the last passage uses the definition of $\gamma_n$. Since $\gamma_n\leq \frac{1}{2}$ we get that 
\begin{align*}
\left(1-\frac{\pi -2}{2\pi}\gamma_n^2\right)\left(\sqrt{1-\frac{\gamma_n^2}{4}}\right)^{-\frac{1}{2}} >0.97.
\end{align*}
Thus, the result follows.
\end{proof}
\section{Conclusion}
In our results, we establish bounds for quantities in terms of $B_p(P_n)$, as defined in equation \eqref{norm B_p}, for all values of $p$ within the open interval $(0, \infty)$. These quantities contain the measure of the set $E$, wherein the absolute value of the polynomial on the unit circle is smaller than one. Corollary \ref{coro0} and Example \ref{exam} shed light on the size of the set $|E|$, which is remarkably small, particularly for polynomial degrees that are sufficiently large. Consequently, a pressing question arises: Can we identify a class of polynomials for which the size of $|E|$ remains constant?

Furthermore,  our results are on the deterministic polynomials; however, it is worth noting that these results can be readily extended to random polynomials by using ideas from \cite{PS, PY, Pr}.
 \section{Acknowledgement}
 The author would like to thank the referees for reviewing the manuscript in detail and giving valuable comments.
The author also would like to express gratitude to Prof. R. Balasubramanian for providing valuable suggestions.
The author is supported by DST, Government of India under the DST-INSPIRE Faculty Scheme with Faculty Reg. No. IFA21-MA 168.

\end{document}